\documentclass[11pt]{amsart}
\usepackage{amsfonts}
\usepackage{dsfont}

\usepackage[a4paper, margin=2cm]{geometry}

\newtheorem{thm}{Theorem}

\newtheorem{lem}[thm]{Lemma}

\newtheorem{conj}[thm]{Conjecture}

\theoremstyle{definition}
\newtheorem{defn}[thm]{Definition}
\theoremstyle{remark}

\begin{document}

\title[]{Binary words containing few abelian squares}

\author{Szil{\'a}rd Zsolt Fazekas}
\address{Akita University\newline 
\indent Japan}
\email{ <szilard.fazekas@gmail.com>}

\author{Adam Mammoliti}
\address{School of Mathematics and Statistics\newline
\indent University of New South Wales\newline
\indent NSW 2052\newline
\indent Australia}
\email{adam.mammoliti@outlook.com.au}

\author{Robert Merca{\c s}}
\address{Loughborough University \newline
\indent UK}
\email{R.G.Mercas@lboro.ac.uk}

\author{Jamie Simpson}%
\address{130 Preston Point Rd \newline
\indent East Fremantle,\newline
\indent WA 6158\newline
\indent Australia}%
\email{jamiesimpson320@gmail.com}%

\begin{abstract} 
In \cite{FS} Fici and Saarela conjectured that a binary word of length $n$ contains at least $\lfloor n/4 \rfloor$ abelian squares.  We slightly extend this conjecture and show that it holds in some special cases. In all other cases we have the following: given a Parikh vector over a two letter alphabet we produce a word with that Parikh vector which we conjecture contains the least possible number of abelian squares.

\end{abstract}

\maketitle

\section{Introduction}
 We use the usual notation for combinatorics of words. A \emph{word} is the concatenation of letters taken from an \emph{alphabet}.  In this paper we use the alphabet $\{a,b\}$, hence our words are \emph{binary}.  A word containing $n$
letters is $w=w[1 \dots n]$, with $w[i]$ being the $i$th letter and
$w[i\dots j]$ being the \emph{factor} beginning at position $i$ and ending at
position $j$.  If $i=1$ then the factor is a \emph{prefix} and if
$j=n$ then the factor is a \emph{suffix}.

The \emph{length} of $w$, written $|w|$, is the
number of letters that $w$ contains and $|w|_a$ and $|w|_b$ are, respectively, the numbers of $a$'s and the number of $b$'s in $w$ (this notation has the obvious extension to larger alphabets). The vector $[|w|_a,|w|_b]$ is the \emph{Parikh vector} of $w$, written $\mathcal{P}(w)$.  If $w=uv$ where $u$ and $v$ are words then $vu$ is a
\emph{conjugate} of $w$.  The \emph{complement} of $w$ is the word formed from $w$ by replacing each $a$ with $b$ and vice versa. The \emph{reverse} of $w$ is the word formed by writing $w$ backwards.  Thus the complement and the reverse of $abaab$ are $babba$ and $baaba$ respectively. 

A factor $w[i+1..i+2p]$ in which $w[i+1..i+p]=w[i+p+1..i+2p]$ is  a \emph{square}. If instead  $\mathcal{P}(w[i+1..i+p])=\mathcal{P}(w[i+p+1..i+2p])$ then $w[i+1..i+2p]$ is an \emph{abelian square}.  In older literature, for instance \cite{FS}, it was a \emph{weak square} \cite{FSP}. Thus the word $abaababa$ contains the squares $aa$, $abab$, $abaaba$, $baba$ which are also abelian squares and the abelian squares $baab$ and $aababa$. Clearly a word, its complement and its reverse each contain the same number of abelian squares.
If $w$ is a binary word then $\theta(w)$ is the number of abelian squares in $w$. The following conjecture is due to Fici and Saarela \cite{FS}. 
\begin{conj}\label{FSconj} Any binary word of length $n$ contains at least $\lfloor \frac{n}{4}\rfloor$ abelian squares.
\end{conj} 
Following Fici and Saarela we say that an abelian square of the form $a^{2i}$ or $b^{2i}$ for some positive integer $i$ is \emph{trivial}. They proved the following theorem.
\begin{thm}\label{FSthm} \cite{FS}
  If a binary word $w$ contains only trivial abelian squares then it satisfies Conjecture \ref{FSconj}.
\end{thm}

However the following theorem shows this result has limited applicability.
\begin{thm}
  If a binary word contains three or more powers of one letter separated by powers of the other letter then it contains a non-trivial abelian square.
  \begin{proof}
     Suppose the word $w$ contains the factor $a^{i_1}b^{j_1}a^{i_2}b^{j_2}a^{i_3}$ where each of $i_1$, $j_1$, $i_2$, $j_2$ and $i_3$ is a positive integer and suppose, without loss of generality, that $j_1 \le j_2$. If $i_2$ is even then $w$  contains the non-trivial abelian square $b^{j_1}a^{i_2}b^{j_1}$ and if $i_2$ is odd then $w$ contains the non-trivial abelian square $ab^{j_1}a^{i_2}b^{j_1}$.
  \end{proof}
\end{thm}

On the same theme we extend the Fici-Saarela conjecture to the following:

\begin{conj}  A binary word of length $n$ contains at least $\lfloor \frac{n}{4}\rfloor$ abelian squares and if it contains exactly $\lfloor \frac{n}{4}\rfloor$ abelian squares then all the abelian squares are trivial.
\end{conj}

We say $M(x,n)$ is the least number of abelian squares that can occur in a binary word $w$ with $|w|=n$ and $|w|_a=x$. This seems  difficult to work out in general but we can obtain precise values for some low values of $x$ and make conjectures about others.

\section{ $x =0$ and $x=1$}

\begin{thm}\label{t1} For positive integers $n$ we have
\begin{equation*}
M(0,n)=\lfloor n/2 \rfloor.
\end{equation*}
\begin{proof} The word $b^n$ contains abelian squares $b^2, b^4, \dots, b^{2\lfloor n/2\rfloor}$.  Hence result.
\end{proof}
\end{thm}

For $x=1$ we need the following easily established identity. If $n$ is an integer then
\begin{equation}\label{identity} \left\lfloor\frac{\lceil\frac{n-1}{2}\rceil }{2} \right \rfloor = \lfloor n/4 \rfloor\end{equation}

\begin{thm} For positive integers $n$ we have 
\begin{equation*}
                                       M(1,n)=\lfloor n/4 \rfloor.
\end{equation*} 
The only words with one $a$ and exactly $M(1,n)$ abelian squares are $b^{\lceil(n-1)/2\rceil}ab^{\lfloor(n-1)/2\rfloor}$ and its reverse.
\begin{proof}
  If $w$ is a binary word with $|w|=n$ and $|w|_a=1$ then it has the form $b^iab^{n-1-i}$ where $0 \le i \le n-1$.  Clearly either $i$ or $n-1-i$ is at least $\lceil(n-1)/2\rceil$, so $w$ contains the factor $b^{\lceil(n-1)/2\rceil}$, so as in Theorem~\ref{t1} it contains at least $\lfloor\lceil(n-1)/2\rceil/2\rfloor$ abelian squares. Hence $w$ contains $b^{\lceil(n-1)/2\rceil}$ and no larger power of $b$ if and only if $w$ is $b^{\lceil(n-1)/2\rceil}ab^{\lfloor(n-1)/2\rfloor}$ or its reverse.   The statement of the theorem then follows from (\ref{identity}). \end{proof} \end{thm}

\section{ $x=2$}
Table 1 gives length $n$ words containing 2 $a$'s and containing the least possible number of abelian squares.
\begin{table}
$\begin{array}{lcl}
\text{n} & \text{Number} & \text{Words achieving this}\\
2& 1& {aa}\\
3& 0& {aba}\\
4& 1& {abab,baba}\\
5& 1& {abbba}\\
6& 2& {ababbb, abbbab, babbba, bbbaba}\\
7& 2& {abbbabb, abbbbba, bbabbba}\\
8& 3& {ababbbbb, abbbabb, abbbbbab, babbbbba, bbbabbba, bbbbbaba}\\
9& 3& {abbbbbabb, abbbbbbba, bbabbbbba}\\
10& 4& {ababbbbbbb, abbbabbbbb, abbbbbabbb, abbbbbbbab, babbbbbbba, 
bbbabbbbba, bbbbbabbba, bbbbbbbaba}\\
11& 4& {abbbbbabbbb, abbbbbbbabb, abbbbbbbbba, bbabbbbbbba, 
bbbbabbbbba}\\
\end{array}$
\caption{Length $n$ words containing two $a$'s and the least possible number of distinct abelian squares.}
\end{table}

If $n=2$, then $w =aa$ has 1 abelian square. For words of length at least 3 we have the following.
\begin{thm}\label{t:2as}
Let $n \geq 3$ and $w$ be a binary word of length $n$ with $2$ $a$'s. Then $w$ contains at least $\left\lfloor \frac{n-2}{2}\right\rfloor$ abelian squares and $w$ contains exactly that many 
abelian squares if it or its reverse is of the form
$b^m a b^{n-m-2}a$ with $n-m$ odd and, if $m \geq \left\lfloor \frac{n-1}{2}\right\rfloor$ holds, then $m$ must also odd.  Thus, for $n \ge 3$,
\begin{equation*}
                                       M(2,n)=\lfloor (n-2)/2\rfloor.
\end{equation*} 
\end{thm}
\begin{proof}
Let $w$ be a binary word of length $n$ containing exactly two $a$'s. We can represent $w$ uniquely as:
\[ w = b^i a b^j a b^k \]
where $i, j, k \ge 0$ and $i + j + k + 2 = n$.

Let $S$ be the total number of distinct abelian squares in $w$. We partition this into $S_0$ (trivial squares containing no $a$'s) and $S_2$ (non-trivial squares containing both $a$'s). Thus, $S = S_0 + S_2$.

A trivial abelian square in $w$ is a factor of the form $b^{2m}$ for $m \ge 1$. The total number of such squares is strictly bounded by the longest contiguous run of $b$'s in the word. Therefore:
\[ S_0 = \max \left( \lfloor i/2 \rfloor, \lfloor j/2 \rfloor, \lfloor k/2 \rfloor \right) \]

A non-trivial abelian square must contain both $a$'s. Thus, it must be a factor of the form $f(x, z) = b^x a b^j a b^z$. For $f(x, z)$ to be a factor of $w$, we must have $0 \le x \le i$ and $0 \le z \le k$.

For $f(x, z)$ to be an abelian square, two conditions must be met:
\begin{itemize}
    \item The total length $|f(x, z)| = x + j + z + 2$ must be even. This implies $x + z \equiv j \pmod 2$.
    
    \item Both halves of equal length must contain exactly one `a'. This means the midpoint of $f(x,z)$ must land within the inner $b^j$ block. Let the left half consist of $b^x a b^{y_1}$ and the right half $b^{y_2} a b^z$, where $y_1 + y_2 = j$. For these halves to be abelian equivalent, their Parikh vectors must match, and since each already has one $a$, we equate the $b$'s:
    \[ x + y_1 = y_2 + z \]
    Substituting $y_2 = j - y_1$, we get $2y_1 = j + z - x$. Because $y_1$ is the length of a prefix of the $b$'s inside the middle block, it must be an integer satisfying $0 \le y_1 \le j$:
    \[ 0 \le \frac{j + z - x}{2} \le j \implies -j \le z - x \le j \implies |x - z| \le j \]
\end{itemize}

Thus, $S_2$ is exactly the cardinality of the integer pair set:
\[ \mathcal{I} = \{ (x, z) \in \mathbb{Z}^2 \mid 0 \le x \le i,\ 0 \le z \le k,\ x+z \equiv j \pmod 2,\ |x-z| \le j \} \]

We must prove that $S_0 + S_2 \ge \lfloor \frac{i+j+k}{2} \rfloor$. By symmetry, we can assume without loss of generality that $i \ge k$.

Let us isolate a subset of valid pairs in $\mathcal{I}$ to establish a floor for $S_2$. Consider only the pairs where $z = 0$. The valid $x$ values corresponding to $z=0$ are integers $x$ such that $0 \le x \le \min(i, j)$ and $x \equiv j \pmod 2$. 

Let $m = \min(i, j)$. The number of integers in $[0, m]$ sharing the parity of $j$ is given by $\lfloor \frac{m+1}{2} \rfloor$ if $j$ is odd, and $\lfloor \frac{m}{2} \rfloor + 1$ if $j$ is even. In both cases, this count is at least $\lfloor \frac{m+1}{2} \rfloor$. Therefore, $S_2 \ge \lfloor \frac{\min(i,j)+1}{2} \rfloor$.

Now we analyze the total squares $S_0 + S_2$ based on which $b$ block dominates:

\begin{itemize}
    \item \textbf{The middle block dominates ($j \ge i \ge k$).} \\
    Here, $S_0 = \lfloor j/2 \rfloor$. Because $j \ge i$, the boundary condition $|x-z| \le j$ is trivially satisfied for all $x \in [0, i]$ and $z \in [0, k]$ (since $|x-z| \le i \le j$). Thus, $\mathcal{I}$ includes every pair in the grid $[0, i] \times [0, k]$ that satisfies the parity condition $x+z \equiv j \pmod 2$. For any fixed $z$, at least $\lfloor\frac{i+1}{2}\rfloor$ of the $x$ values will match the parity, giving at least $\lfloor \frac{i+1}{2} \rfloor$ valid $x$ for each $z$. Summing over all $k+1$ values of $z$, we get $S_2 \ge (k+1)\lfloor \frac{i+1}{2} \rfloor$. \\
    Summing the total: 
    \[ S_0 + S_2 \ge \lfloor j/2 \rfloor + (k+1)\lfloor \frac{i+1}{2} \rfloor \ge \lfloor \frac{i+j+k}{2} \rfloor \]
    satisfying the bound.

    \item \textbf{The outer block dominates ($i \ge j$).} \\
    Here, $S_0 = \lfloor i/2 \rfloor$. Using our $z=0$ baseline bound from earlier, since $i \ge j$, we have $\min(i, j) = j$. Therefore, $S_2 \ge \lfloor \frac{j+1}{2} \rfloor$. \\
    Summing the total we get: 
    \[ S_0 + S_2 \ge \lfloor i/2 \rfloor + \lfloor \frac{j+1}{2} \rfloor \ge \lfloor \frac{i+j}{2} \rfloor \]
    If $k = 0$ (meaning one $a$ is at the boundary), this exactly achieves the minimum $S \ge \lfloor \frac{i+j}{2} \rfloor = \lfloor \frac{n-2}{2} \rfloor$. If $k \ge 1$, then setting $z=0$ ignores many valid pairs. Since $i\ge k$, every increment of $z$ up to $k$ adds at least one valid pair, either $(z,z)$ when $j$ is even, or $(z\pm 1,z)$ (covering also the cases $z=0$ and $z=k=i$) when $j$ is odd. Each such pair increments $S_2$, pushing the total sum to be greater than the $\lfloor \frac{i+j+k}{2} \rfloor$ minimum.
\end{itemize}

In all configurations, the word $w$ must contain at least $\lfloor (n-2)/2 \rfloor$ distinct abelian squares. Furthermore, the bounds show that this absolute minimum is only achievable when $k=0$ (or $i=0$, respectively), pushing one $a$ to the word's boundary and matching the exact parity conditions stated above. 
\end{proof}

\section{$x=3$} 
Table \ref{table 2}  gives length $n$ words containing exactly 3 $a$'s and containing the least possible number of abelian squares.\\

\begin{table} \label{table 2}
$\begin{array}{lcl}
\text{n} & \text{Number} & \text{Words achieving this}\\
5&   1&  {baaab}  \\
6&   2&    {aaabbb, aabbba, ababab, abbbaa, baaabb, bababa, bbaaab, bbbaaa} \\ 
7&   2&    {baaabbb, bababab, bbaaabb, bbbaaab} \\ 
8&   2&  {bbaaabbb, bbbaaabb}  \\
9&   2&   {bbbaaabbb} \\ 
10&  3&   {bbaaabbbbb, bbbaaabbbb, bbbbaaabbb, bbbbbaaabb} \\ 
11&  3&   {bbbaaabbbbb, bbbbaaabbbb, bbbbbaaabbb}  \\
12&  3&    {bbbbaaabbbbb, bbbbbaaabbbb} \\ 
13&  3&   {bbbbbaaabbbbb}  \\
14&   4&    {bbbbaaabbbbbbb, bbbbbaaabbbbbb, bbbbbbaaabbbbb, bbbbbbbaaabbbb} \\ 
15&   4&    {bbbbbaaabbbbbbb, bbbbbbaaabbbbbb, bbbbbbbaaabbbbb}  \\
16&   4&   {bbbbbbaaabbbbbbb, bbbbbbbaaabbbbbb}  \\
17&   4&    {bbbbbbbaaabbbbbbb} \\ 
  \end{array}  $
  \caption{Length $n$ words containing three $a$'s and the least possible number of distinct abelian squares}\label{3a }
  
 \end{table} 

\begin{thm} (a) If $w$ is a word of length $n$ on the alphabet $\{a,b\}$ with $|w|_a=3$  then it contains at least $$\lfloor \frac{\lceil \frac{n-3}{2} \rceil}{2}\rfloor+1=\lfloor \frac{\lceil \frac{n+1}{2} \rceil}{2}\rfloor=\left\lfloor \frac{n+2}{4}\right\rfloor$$ abelian squares.  This bound is achieved by \begin{equation} \label{best 3a word}  b^{\lfloor (n-3)/2 \rfloor}a^3b^{\lceil (n-3)/2 \rceil}. \end{equation}
(b) The bound is also achieved by other words.  The full set is as follows,\\
if $n \equiv 0 \pmod{4}$ then $w=b^{(n-4)/2}a^3b^{(n-2)/2}$ or its reverse,\\
if  $n \equiv 1 \pmod{4}$ then $w=b^{(n-3)/2}a^3b^{(n-3)/2}$,\\
if $n \equiv 2 \pmod{4}$ then $w=b^{(n-6)/2}a^3b^{n/2}$, $b^{(n-2)/2}a^3b^{(n-4)/2}$ or one of their reverses,\\
if $n \equiv 3 \pmod{4}$ then $w=b^{(n-3)/2}a^3b^{(n-3)/2}$, $b^{(n-1)/2}a^3b^{(n-5)/2}$  or its reverse.

\begin{proof}
The cases when $n\leq 10$ can easily be checked.
So we assume that $n \ge 11$.  Let $w$ be a binary word of length $n$ with $3$ a's at positions $p$, $q$ and $r$ where $$1 \le p<q<r \le n.$$ The prefix $w[1..r-1]$  contains exactly two $a$'s. If its length is at least $\lceil \frac{n+5}{2} \rceil$  then by Theorem \ref{t:2as} it contains at least
  \[ \left\lfloor\frac{\lceil \frac{n+5}{2} \rceil-2}{2}\right\rfloor = \left\lfloor\frac{\lceil \frac{n+1}{2} \rceil}{2}\right\rfloor \]
abelian squares and contains exactly that many only if $w[1..r-1]$ is $b^m a b^{n-m-2}a$ or its reverse, where $n-m$ is odd. If $w[1..r-1]=b^m a b^{n-m-2}a$, then $w[1..r]=b^m a b^{n-m-2}aa$ which contains the abelian square $a^2$ and if $w[1..r-1]=a b^{n-m-2} a b^{m}$, then $w[1..r]=a b^{n-m-2} a b^{m}a$ contains both $ab^m a$ and $bab^ma$, one of which is an abelian square. In either case, $w[1..r]$ contains an abelian square not in $w[1..r-1]$ and so $w$ contains more than $\left\lfloor\frac{\lceil \frac{n+1}{2} \rceil}{2}\right\rfloor $ abelian squares 
  and the statement of the theorem holds. 
  
 So we assume instead that $ r-1 < \lceil \frac{n+5}{2} \rceil$ so that
  \[ r \le \left\lceil \frac{n+5}{2} \right\rceil. \] 
  Now consider the suffix $w[p+1..n]$ which contains exactly two $a$'s. By similar reasoning to that above we conclude that if a counterexample exists we must have 
  \[ p \ge n - \left\lceil \frac{n+5}{2} \right\rceil +1 \]
and  so  all three $a$'s lie in the factor $w'=w[ n - \lceil \frac{n+5}{2}\rceil+1..\lceil \frac{n+5}{2} \rceil] $.  This factor has length 5 or 6 depending on the parity of $n$. 
  Let $y = ab^iab^ja$ be the smallest factor of $w'$ containing the three $a$'s, where $i$ and $j$ are non-negative. As $w'$ has length at most $6$, $y$ also has length at most $6$ and so $i+j \leq 3$. 
  Then $w= b^s y b^t$, where given that $y$ is contained in $w'$ means $s,t \geq \left\lfloor \frac{n-5}{2}\right\rfloor$. 
 To complete the proof we must consider the possible positions of the three $a$'s in $y$. We have two cases.
  
\noindent \textit{Case 1: } $i=0=j$ and so $y=aaa$. Then $w=b^sa^3b^t$, with $n-3=s+t$ which contains $\left\lfloor \frac{\max\{s,t\}}{2} \right\rfloor +1$ abelian squares. So the number of abelian squares is minimised by the word in (\ref{best 3a word}) and the other words in part (b) above.
\\

\noindent \textit{Case 2:} $i$ and $j$ are not both 0. 
As we have assumed that $|w|\geq 11$, $s,t \geq 3$ and 
$w = b^sa b^iab^jab^t$ contains the abelian squares $b^i a .b^i a$ and $ab^j .ab^j$ along with  $b^{i} a b^j .a b^{i+j}$ if $i\neq 0$ and $b^{i+j}a b^i ab^j$  if $j\neq 0$, all of which are necessarily distinct.
Also $w$ has $\left\lfloor \frac{\max\{s,t\}}{2} \right\rfloor $ abelian squares consisting of only $b$'s.
It follows that $w$ has at least
\[
\left\lfloor \frac{ \left\lfloor  \frac{n-5}{2} \right\rfloor}{2}\right\rfloor +3 =\left\lfloor \frac{ \left\lfloor  \frac{n-5}{2} \right\rfloor +4}{2}\right\rfloor+1 =
\left\lfloor \frac{ \left\lfloor  \frac{n+3}{2} \right\rfloor }{2}\right\rfloor+1 > \left\lfloor \frac{ \left\lceil  \frac{n+1}{2} \right\rceil }{2}\right\rfloor 
\]
abelian squares.

\noindent
These two cases cover all the possible values of $i$ and $j$ so we are done.
%
\end{proof}
\end{thm}

\section{Case 5: $3 < x <n-2$} 

    Table 3 gives some examples of length 18 words with specified numbers of $a$'s and $b$'s which contain the least possible numbers of abelian squares.   Ignore the cases where we don't have $3<x<n-2$ as these are covered by earlier results (call them \emph{boundary cases}). Other cases have the form $a^hb^ia^jb^k$. These are not the only words achieving the bound but we believe there's always an example with this form. Below we make a conjecture about the values of $h$, $i$, $j$ and $k$ for given $|w|_a$ and $|w|_b$.\\
   
\begin{table}   
      $\begin{array}{ccl}\\
 \text{Number of $a$'s} & \text{Number of abelian squares} & \text{Sample word} \\
 0 & 9 & b^{18}\\
 1 & 4 & b^9ab^8\\
 2 & 8 & abab^{15}\\
 3 & 5 & b^6a^3b^9\\
 4 & 5 & ab^7a^3b^7\\
 5 & 4 & a^2b^7a^3b^6\\
 6 & 5 & ab^7a^5b^5\\
 7 & 5 & a^2b^7a^5b^4\\
 8 & 5 & a^3b^7a^5b^3\\
 9 & 4 & a^4b^5a^5b^4\\
 10 & 5 & a^3b^5a^7b^4\\
 11 & 5 & a^4b^5a^7b^2\\
 12 & 5 & a^5b^5a^7b\\
 13 & 4 & a^5b^3a^7b^2\\
 14 & 5 & a^7b^3a^7b\\
 15 & 5 & a^9b^3a^6\\
 16 & 8 & a^{15}bab\\
 17 & 4 & a^9ba^8\\
 18 & 9 & a^{18}
 \end{array}$
 \caption{Length 18 words with specified number of $a$'s and least possible number of abelian squares.}
 \end{table}
   
   \begin{defn}
   If $n$ is a positive integer with $n \ge 4$ then an \emph{effective partition} of $n$ is the ordered pair $e(n)=[p,q]$ where $p$ and $q$ are positive integers satisfying\\
   (a) $p+q=n$,\\
   (b) $q>p$,\\
   (c) $q$ is odd, and\\
   (d) $q$ is as small as possible consistent with (a), (b)  and (c).\\
   \end{defn}
   These conditions uniquely define $e(n)$. For example, $e(10)=[3,7]$.
   
 \begin{thm} \label{values of p and q} Using the notation of the definition and for $n\ge 4$,
  \begin{eqnarray*}
    p &=& n-2\lfloor \frac{n+2}{4} \rfloor-1 \\
    q &=& 2\lfloor \frac{n+2}{4} \rfloor+1 \\
  \end{eqnarray*}
 \begin{proof}
   We must show that the four parts of the definition of $e$ are satisfied.\\
   (a) Clearly $p+q =n$,\\
   (b) $q-p= 4\lfloor \frac{n+2}{4} \rfloor-n+2 > 4(\frac{n+2}{4}-1)-n+2=0$ so $q>p$,\\
   (c) $q$ is clearly odd,\\
   (d) if $q$ is replaced with a smaller odd number $q'$ then we'd have $q' \le q-2$ to satisfy (c) and $p$ would be replaced by $p'$ where $p'\ge p+2$ to satisfy (a), but then
   \begin{eqnarray*}
     q'-p' &\le & 2\lfloor \frac{n+2}{4} \rfloor-1-(n-2\lfloor \frac{n+2}{4} \rfloor+1) \\
      &=& 4\lfloor \frac{n+2}{4} \rfloor -n-2\\
      &\le & 0
   \end{eqnarray*} so we'd have $q' \le p'$ contradicting (b).  So $q$ cannot be replaced with a smaller number without contradicting one of (a), (b) and (c) so (d) is satisfied.
 \end{proof}
 \end{thm}
 
 We'll need the following Lemma.
 \begin{lem}\label{effective word inequality} If $m$ and $n$ are positive integers then
 \[\lfloor \frac{m+2}{4}\rfloor + \lfloor\frac{n+2}{4}\rfloor \ge \lfloor \frac{m+n}{4}\rfloor.
 \]
 \begin{proof}
   Clearly changing the value of $m$ or $n$ by 4 will not alter the direction of the inequality.  We may assume, therefore, that $m$ and $n$ come from the set $\{0,1,2,3\}$.  It is then easy to check all possibilities.
 \end{proof}
 \end{lem}

    We says that $w$ is an \emph{effective word} if it has the form $a^hb^ia^jb^k$ where $[h,j]$ and $[k,i]$ are the effective partitions of $|w|_a$ and $|w|_b$, with $h<j$ and $k<i$.

   \begin{thm}
   (a) The only abelian squares inside the effective word $w=a^hb^ia^jb^k$  are trivial.\\
   (b) The number of distinct abelian squares in the word is $\lfloor \frac{|w|_a+2}{4} \rfloor +\lfloor \frac{|w|_b+2}{4} \rfloor$.\\
   (c) This number  is at least $\lfloor |w|/4 \rfloor$.
   \begin{proof}
     (a) Write $w$ as $A_1B_1A_2B_2$ where $A_1=a^h$ and so on. There can be no abelian square beginning in $A_1$ and ending in $B_1$ as one half would be a power of a single letter while the other half would contain at least one occurrence of the other letter.  Similarly there is no abelian square beginning in $B_1$ and ending in $A_2$, nor one beginning in $A_2$ and ending in $B_2$. A factor beginning in $A_1$ and ending in $A_2$ contains the whole of odd length $B_1$ so cannot be an abelian square.  Similarly there is no abelian square beginning in $B_1$ and ending in $B_2$. Finally a factor beginning in $A_1$ and ending in $B_2$ will contain either the whole of $B_1$ or the whole of $A_2$ in one of its halves.  This cannot be an abelian square since $|A_2|>|A_1|$ and $|B_1|>|B_2|$.  Hence the only abelian squares in $w$ lie entirely inside one of $A_1$, $B_1$, $A_2$ and $B_2$ and are therefore trivial.\\
     
     \noindent (b) Since $j>h$ then number of abelian squares of the form $a^{2m}$ is $\lfloor j/2 \rfloor$. Using Theorem \ref{values of p and q} we see that this is $\lfloor \frac{|w|_a+2}{4} \rfloor.$ Similarly the number of abelian squares of the form $b^{2m}$ is  $\lfloor \frac{|w|_b+2}{4} \rfloor$, and part (b) of the theorem follows. \\
     
     \noindent (c) This follows from Lemma \ref{effective word inequality}.
   \end{proof}
   \end{thm}
   
   \begin{conj} If $w$ is a non-boundary binary word,  $e(|w|_a)=[i,k]$ and $e(|w|_b)=[j,k]$ then $w$ contains at least $\lfloor i/2 \rfloor +\lfloor j/2 \rfloor$ abelian squares.
   \end{conj}
   
\section{Discussion} We have found the minimum number of abelian squares in a word of length $n$ on alphabet $\{a,b\}$ which contains 0,1,2 or 3 occurrences of the letter $a$, and by complementation,  $n$, $n-1$, $n-2$ or $n-3$ occurrences.  We would like to settle the cases in between. We used the two $a$ bound to establish the three $a$ bound.  This suggests that an induction might be possible. Another approach might be to show that as soon as a word  contains one non-trivial abelian square it must contain sufficiently many others to reach the $\lfloor n/4 \rfloor$ bound. Coupled with Theorem \ref{FSthm} this would lead to a proof of the conjecture.

There are a number of open questions related to Fici and Saarela's Conjecture \ref{FSconj}. See, for instance, the surveys \cite{JS} and \cite{FP}.  One variation is to count inequivalent abelian squares rather than distinct abelian squares. Two abelian squares are equivalent if they have the same Parikh vectors. Thus $abba$ and $abab$ are distinct but equivalent.  Fraenkel and Simpson \cite{FSP} used a straightforward approach to show that the number of inequivalent abelian squares in a circular binary word of length $2k+2$ is at least $k$, the bound being attained by $(ab)^{k+1}$.

\bibliography{}

 \end{document}